\documentclass[12pt]{amsart}
\usepackage{amssymb,amsmath,amsfonts,latexsym,amscd}
\usepackage{bm}
\newcommand{\newsection}[1]{\setcounter{equation}{0} \section{#1}}
\newcommand{\vp}{\varphi}

\newcommand{\clb}{\mathcal{B}}

\newcommand{\cld}{\mathcal{D}}
\newcommand{\cle}{\mathcal{E}}

\newcommand{\clh}{\mathcal{H}}
\newcommand{\clk}{\mathcal{K}}
\newcommand{\cll}{\mathcal{L}}
\newcommand{\clm}{\mathcal{M}}

\newcommand{\clo}{\mathcal{O}}

\newcommand{\clw}{\mathcal{W}}

\newcommand{\raro}{\rightarrow}

\newcommand{\be}{\begin{equation}}
\newcommand{\ee}{\end{equation}}
\newcommand{\ben}{\begin{eqnarray*}}
\newcommand{\een}{\end{eqnarray*}}
\newcommand{\NI}{\noindent}
\newcommand{\bi}{\begin{itemize}}
\newcommand{\ei}{\end{itemize}}

\newcommand{\D}{\mathbb{D}}

\newtheorem{Theorem}{\sc Theorem}[section]
\newtheorem{Lemma}[Theorem]{\sc Lemma}
\newtheorem{Proposition}[Theorem]{\sc Proposition}
\newtheorem{Corollary}[Theorem]{\sc Corollary}

\newtheorem{Example}[Theorem]{\sc Example}
\newtheorem{Remark}[Theorem]{\sc Remark}
\newtheorem{Note}[Theorem]{\sc Note}
\newtheorem{Question}{\sc Question}
\newtheorem{ass}[Theorem]{\sc Assumption}
\theoremstyle{definition}
\newtheorem{Definition}[Theorem]{\sc Definition}

\newcommand{\bt}{\begin{Theorem}}
\def\beginlem{\begin{Lemma}}
\def\beginprop{\begin{Proposition}}
\def\begincor{\begin{Corollary}}
\def\begindef{\begin{Definition}}
\def\beginexamp{\begin{Example}}
\def\beginrem{\begin{Remark}}
\def\beginq{\begin{Question}}
\def\beginass{\begin{ass}}
\def\beginnote{\begin{Note}}
\newcommand{\et}{\end{Theorem}}
\def\endlem{\end{Lemma}}
\def\endprop{\end{Proposition}}
\def\endcor{\end{Corollary}}
\def\enddef{\end{Definition}}
\def\endexamp{\end{Example}}
\def\endrem{\end{Remark}}
\def\endq{\end{Question}}
\def\endass{\end{ass}}
\def\endnote{\end{Note}}

\newcommand{\cls}{\mathcal{S}}

\newcommand{\z}{\bm{z}}
\newcommand{\w}{\bm{w}}
\newcommand{\Z}{\mathbb{Z}_+^n}

\begin{document}

\title[contractively embedded invariant subspaces]
{contractively embedded invariant subspaces}

\author[Gorai]{Sushil Gorai}
\address{Department of Mathematics and Statistics, Indian Institute of Science Education and Research Kolkata, Mohanpur 741 246, West
Bengal, India} \email{sushil.gorai@iiserkol.ac.in}

\author[Sarkar]{Jaydeb Sarkar}
\address{Indian Statistical Institute, Statistics and Mathematics Unit, 8th Mile, Mysore Road, Bangalore,
560059, India} \email{jay@isibang.ac.in, jaydeb@gmail.com}


\keywords{Invariant subspaces, de Branges-Rovnyak spaces, Hardy
space, reproducing kernel Hilbert spaces, multipliers, bounded
analytic functions}

\dedicatory{To Joseph A. Ball with gratitude and admiration}

\subjclass[2000]{46C07, 46E22, 47A13, 47A15, 47B32}

\begin{abstract}
This paper focuses on representations of contractively embedded
invariant subspaces in several variables. We present a version of
the de Branges theorem for $n$-tuples of multiplication operators by
the coordinate functions on analytic reproducing kernel Hilbert
spaces over the unit ball $\mathbb{B}^n$ and the Hardy space over
the unit polydics $\mathbb{D}^n$ in $\mathbb{C}^n$.
\end{abstract}

\maketitle

\section{Introduction}

The theory of contractively embedded invariant and co-invariant (not
necessarily closed) subspaces for the shift operator on the Hardy
space was initiated by L. de Branges. This theory was laid out more
systematically in the mid 60's by de Branges and Rovnyak (see the
monograph by de Branges and Rovnyak \cite{dR}). The de Branges and
Rovnyak's approach to the theory of contractively embedded invariant
and co-invariant subspaces for shift operators on reproducing kernel
Hilbert spaces has proved very fruitful in analysing operator and
function theoretic problems. As is well known, it was this theory
that led de Branges to the affirmative solution of the Bieberbach
conjecture \cite{LdB}.

The purpose of this note is to analyze the structure of
contractively embedded (not necessarily closed) invariant subspaces
for tuples of multiplication operators by the coordinate functions
on reproducing kernel Hilbert spaces in several variables. Recall
that a Hilbert space $\clh$ is said to be contractively embedded in
a Hilbert space $\clk$ if $\clh$ is a vector subspace of $\clk$ and
the inclusion map $i_{\clh} : \clh \hookrightarrow \clk$ is a
contraction. Obviously, the latter condition is equivalent to
\[
\|f\|_{\clk} \leq \|f\|_{\clh},
\]
for all $f \in \clh$, where $\|\cdot\|_{\clh}$ and
$\|\cdot\|_{\clk}$ denotes the norms on $\clh$ and $\clk$,
respectively. It follows in particular that a closed subspace of a
Hilbert space is contractively (or isometrically, as an embedding)
embedded in the larger Hilbert space.

Now let $\clk$ be a Hilbert space, and let $\clh$ be a Hilbert space
that is contractively embedded in $\clk$. Let $(T_1, \ldots, T_n)$
be an $n$-tuple of commuting bounded linear operators on $\clk$,
that is,
\[
T_i T_j = T_j T_i,
\]
for all $i, j = 1, \ldots, n$. Let $\clh$ be an \textit{invariant
subspace} for $(T_1, \ldots, T_n)$, that is,
\[
T_i \clh \subseteq \clh,
\]
for all $i = 1, \ldots, n$. Suppose that $T_i|_{\clh}$ is bounded on
$\clh$, that is, there exists $M > 0$ such that
\[
\|T_i f\|_{\clh} \leq M \|f\|_{\clh},
\]
for all $f \in \clh$ and $i = 1, \ldots, n$. Then clearly
$(T_1|_{\clh}, \ldots, T_n|_{\clh})$ is an $n$-tuple of commuting
bounded linear operators on $\clh$. The question of interest here is
to represent $\clh$ in terms of the (algebraic or analytic
properties of the) tuple $(T_1, \ldots, T_n)$.

We pause now to examine one concrete example of the above invariant
subspace problem. Following standard notation, let $H^2(\D)$ denote
the Hardy space over the unit disc $\D$. Let $M_z$ on $H^2(\D)$ be
the multiplication operator by the independent variable $z$, that
is,
\[
(M_z f)(w) = w f(w),
\]
for all $f \in H^2(\D)$ and $w \in \D$. It follows that $M_z$ is a
shift of multiplicity one (see Section 3). Let $\clh$ be a Hilbert
space contractively embedded in $H^2(\D)$, and let $M_z \clh
\subseteq \clh$. If $M_z|_{\clh}$ is an isometry on $\clh$, then the
celebrated theorem of de Branges says that there is a function $\vp
\in H^\infty(\D)$ such that $\|\vp\|_{\infty} \leq 1$ and
\[
\clh = \vp H^2(\D).
\]
Recall that $H^\infty(\D)$ is the Banach algebra of all bounded
analytic functions on the unit disc $\D$ equipped with the supremum
norm \cite{NF}. Here the norm on $\clh$ is the range norm induced by
the injective multiplier $M_{\vp}$ on $H^2(\D)$, that is,
\[
\|\vp f\|_{\clh} = \|f\|_{H^2(\D)},
\]
for all $f \in H^2(\D)$ (cf. Section 3 in \cite{Don1} and Theorems
3.5 and 3.7 in \cite{T}). In this context, we refer the reader to
the beautiful survey by Ball and Bolotnikov \cite{BB-sur} on de
Branges-Rovnyak spaces in both one and several variables, the
monographs by Fricain and Mashreghi \cite{FM}, Sarason \cite{Sar,
Don1}, Niko\'{l}ski\u{i} and Vasyunin \cite{NV}, Sand \cite{Sand}
and Timotin \cite{T}. Also see Singh and Thukral \cite{DV} and Sahni
and Singh \cite{SD}. Another important and relevant piece of work is
due to Ball and Kriete \cite{BK} and Crofoot \cite{Cro}. The reader
can also see the papers by Chevrot, Guillot and Ransford \cite{CGR},
Costara and Ransford \cite{CR} and Sarason \cite{Don} in connection
with the de Branges-Rovnyak models and (generalized) Dirichlet
spaces.

A natural question is now to ask for a similar representations of
contractively embedded invariant subspaces for (tuples of)
multiplication operators by the coordinate function(s) within the
framework of analytic reproducing kernel Hilbert spaces \cite{Aro}
in one and several variables.

In Theorems \ref{th:gen-inc}, \ref{th-MT2} and \ref{th-MT3}, we
present a solution for this problem in the setting of commuting row
contractions on Hilbert spaces and analytic Hilbert spaces (see the
definition in Section 2) and tuples of shift operators on
vector-valued Hardy spaces over the unit polydisc $\D^n$ in
$\mathbb{C}^n$, respectively.

The proofs of Theorems \ref{th:gen-inc} and \ref{th-MT2} involve a
careful adaptation of techniques used in \cite{JS2} and \cite{JS1}.
Whereas the setting and the proof of our invariant subspace theorem
for the shift on the Hardy space over the polydisc, Theorem
\ref{th-MT3}, is closely related to the recently initiated work
\cite{MMSS} on the classification of (closed) invariant subspace
problem for the Hardy space in several variables.

A somewhat more intriguing and complex problem is the classification
of contractively embedded invariant subspaces which admit a
co-invariant complemented subspace. Note that an important aspect of
the de Branges-Rovnyak theory is the complementations of invariant
subspaces of the Hardy space: A contractively embedded invariant
subspace for $M_z$ on $H^2(\D)$ is complemented in $H^2(\D)$ by an
$M_z^*$-invariant (not necessarily closed) subspace (cf. Subsection
3.4 in \cite{T}). We postpone the general discussion on complemented
invariant subspaces for a future paper and refer the reader to the
papers by Ball, Bolotnikov and Fang \cite{BBF2, BBF1, BBF}, Ball,
Bolotnikov and ter Horst \cite{BBS1, BBS} and Benhida and Timotin
\cite{BT} for related results in the setting of Drury-Arveson space
\cite{Ar}.

For the remainder, we adapt the following notations: $\z$ denotes
the element $(z_1, \ldots, z_n)$ in $\mathbb{C}^n$, $z_i \in
\mathbb{C}$, $\D^n = \{\z \in \mathbb{C}^n: |z_i| < 1, i = 1,
\ldots, n\}$, $\mathbb{B}^n = \{\z \in \mathbb{C}^n:
\|\z\|_{\mathbb{C}^n} < 1\}$ and
\[
\mathbb{Z}_+^n = \{\bm{k} = (k_1, \ldots, k_n) : k_i \in
\mathbb{Z}_+, i = 1, \ldots, n\}.
\]
Also for each multi-index $\bm{k} \in \mathbb{Z}_+^n$, commuting
tuple $T = (T_1, \ldots, T_n)$ on a Hilbert space $\clh$, and $\z
\in \mathbb{C}^n$ we denote
\[
T^{\bm{k}} = T_1^{k_1} \cdots T_n^{k_n} \quad \text{and} \quad
\z^{\bm{k}} = z_1^{k_1} \cdots z_n^{k_n}.
\]

\newsection{Row contractions and reproducing kernel Hilbert spaces}

Let $n$ be a natural number, and let $\clh$ be a Hilbert space. A
commuting tuple of bounded linear operators $(T_1, \ldots, T_n)$
acting on $\clh$ is called a \textit{row contraction} if the row
operator $(T_1, \ldots, T_n) : \clh^n \raro \clh$ defined by
\[
(T_1, \ldots, T_n) \begin{bmatrix} h_1 \\ \vdots \\ h_n
\end{bmatrix} = T_1 h_1 + \cdots + T_n h_n,
\]
for all $h_1, \ldots, h_n \in \clh$, is a contraction. Evidently,
the tuple $(T_1, \ldots, T_n)$ is a row contraction if and only if
\[
\|T_1 h_1 + \cdots + T_n h_n \|^2 \leq \| h_1\|^2 + \cdots +
\|h_n\|^2,
\]
for all $h_1, \ldots, h_n \in \clh$, or equivalently if
\[
\sum_{i=1}^n T_i T_i^* \leq I_{\clh}.
\]

For a row contraction $T = (T_1, \ldots, T_n)$ on a Hilbert space
$\clh$, we define the \textit{defect operator} and the
\textit{defect space} of $T$ as
\[
D_T = (I_{\clh} - \sum_{i=1}^n T_i T^*_i )^{\frac{1}{2}},
\]
and
\[
\cld_{T} = \overline{\text{ran~}} D_T
\]
respectively. Consider the map $P_T : \clb(\clh) \rightarrow
\clb(\clh)$ defined by
\[
P_{T} (X) = \sum_{i=1}^{n} T_i X T^*_i,
\]
for all $X \in \clb(\clh)$. Clearly, $P_T$ is a completely positive
map. Moreover, since
\[
I_{\clh} \geq P_{T} (I_{\clh}) \geq P^2_{T} (I_{\clh}) \geq \cdots
\geq 0,
\]
it follows that
\[
P_{\infty}(T) = \mbox{SOT} - \mathop{\lim}_{m \raro \infty} P_{T}^m
(I_{\clh}),
\]
exists and $0 \leq P_{\infty}(T) \leq I_{\clh}$. We say that $T$ is
a \textit{pure row contraction} if
\[
P_{\infty}(T) = 0.
\]
Standard examples of pure row contractions are the multiplication
operator tuples by the coordinate functions on the Drury-Arveson
space, the Hardy space, the Bergman space and the weighted Bergman
spaces over $\mathbb{B}^n$. In fact, for each $\lambda \geq 1$, the
multiplication operator tuple $(M_{z_1}, \ldots, M_{z_n})$ on the
reproducing kernel Hilbert space $\clh_{K_{\lambda}}$ is a pure row
contraction, where
\begin{equation}\label{eq-K}
K_\lambda(\z, \w) = (1 - \sum_{i=1}^n z_i \bar{w}_i)^{-\lambda},
\end{equation}
for all $\z, \w \in \mathbb{B}^n$ (cf. Proposition 4.1 in
\cite{JS2}). Note that the Drury-Arveson space $H^2_n$, the Hardy
space $H^2(\mathbb{B}^n)$, the Bergman space $L^2_a(\mathbb{B}^n)$,
and the weighted Bergman space $L^2_{a, \alpha}(\mathbb{B}^n)$, with
$\alpha > 0$, are reproducing kernel Hilbert spaces with kernel
$K_{\lambda}$ for $\lambda = 1, n$, $n+1$ and $n + 1 + \alpha$,
respectively.

Let $\cle$ be a Hilbert space. We identify the Hilbert tensor
product $H^2_n \otimes \cle$ with the $\cle$-valued Drury-Arveson
space $H^2_n(\cle)$, or the $\cle$-valued reproducing kernel Hilbert
space with kernel function
\[
\mathbb{B}^n \times \mathbb{B}^n \ni (\z, \w) \mapsto (1 -
\sum_{i=1}^n z_i \bar{w}_i)^{-1} I_{\cle}.
\]
Then
\[
H^2_n(\cle) = \{ f \in \clo (\mathbb{B}^n, \cle): f(z) =
\sum_{\bm{k} \in \mathbb{Z}^n_+} a_{\bm{k}} z^{\bm{k}}, a_{\bm{k}}
\in \cle, \|f\|^2 : = \sum_{\bm{k} \in \mathbb{Z}_+^n} \frac{ \|
a_{\bm{k}} \|^2}{\gamma_{\bm{k}}} < \infty \},
\]
where $\gamma_{\bm{k}} = \frac{(k_1 + \cdots + k_n)!}{k_1 ! \cdots
k_n!}$ are the multinomial coefficients, $\bm{k} \in \mathbb{Z}_+^n$
(cf. \cite{Ar} and \cite{JL}).

Now let $\clk$ be a Hilbert space, and let $\clh$ be a Hilbert space
that is contractively embedded in $\clk$. Let $T = (T_1, \ldots,
T_n)$ be a pure row contraction on $\clk$. Let
\[
T_j \clh \subseteq \clh,
\]
and let
\[
R_j = T_j|_{\clh},
\]
for all $j = 1, \ldots, n$. Suppose that $R = (R_1, \ldots, R_n)$ is
a row contraction on $\clh$, that is,
\[
\|\mathop{\sum}_{i=1}^n R_i h_i\|^2_{\clh} = \|\mathop{\sum}_{i=1}^n
T_i h_i\|^2_{\clh} \leq \mathop{\sum}_{i=1}^n \|h_i\|^2_{\clh},
\]
for all $h_1, \ldots, h_n \in \clh$. First we claim that $(R_1,
\ldots, R_n)$ is a pure row contraction. Indeed, observe that
\[
i_{\clh} R_j = T_j i_{\clh},
\]
for all $j = 1, \ldots, n$. Then
\[
i_{\clh} R^{\bm{k}} = T^{\bm{k}} i_{\clh},
\]
and so
\[
R^{*\bm{k}} i_{\clh}^* = i_{\clh}^* T^{*\bm{k}},
\]
for all $\bm{k} \in \Z$. This yields
\[
i_{\clh} R^{\bm{k}} R^{*\bm{k}} i_{\clh}^* = T^{\bm{k}} i_{\clh}
i_{\clh}^* T^{*\bm{k}},
\]
for all $\bm{k} \in \Z$, and hence
\[
i_{\clh} P^m_R(I_{\clh}) i_{\clh}^* = P^m_T(i_{\clh} i_{\clh}^*),
\]
for each $m \geq 0$. Since $P^m_T : \clb(\clk) \raro \clb(\clk)$ is
a (completely) positive map and
\[
i_{\clh} i_{\clh}^* \leq I_{\clk},
\]
(recall that $i_{\clh} : \clh \hookrightarrow \clk$ is a
contraction) we obtain that
\[
P^m_T(i_{\clh} i_{\clh}^*) \leq P^m_T(I_{\clk}),
\]
and hence
\[
i_{\clh} P^m_R(I_{\clh}) i_{\clh}^* \leq P^m_T(I_{\clk}),
\]
for all $m \geq 0$. Now for $f \in \clk$ and $m \geq 0$, we compute
\[
\begin{split}
\| P^m_R(I_{\clh})^{\frac{1}{2}} i_{\clh}^* f\|^2_{\clh} & = \langle
P^m_R(I_{\clh}) i_{\clh}^* f, i_{\clh}^* f \rangle_{\clh}
\\
& = \langle i_{\clh} P^m_R(I_{\clh}) i_{\clh}^* f, f \rangle_{\clk}
\\
& \leq \langle P^m_T(I_{\clk}) f, f \rangle_{\clk}.
\end{split}
\]
Since $(T_1, \ldots, T_n)$ is a pure row contraction we see that
\[
\lim_{m \raro \infty} \|P^m_R(I_{\clh})^{\frac{1}{2}} i_{\clh}^*
f\|_{\clh} = 0,
\]
for all $f \in \clk$. On the other hand, since $i_{\clh}$ is
one-to-one we see that $i_{\clh}^* : \clk \raro \clh$ has dense
range, and hence by continuity
\[
SOT - \lim_{m \raro \infty} P^m_R(I_{\clh})^{\frac{1}{2}} = 0.
\]
Since the sequence of positive operators $\{P^m_R(I_{\clh})\}_{m\geq
0}$ is uniformly bounded (by $\|I_{\clh}\| = 1$) we obtain that
\[
SOT - \lim_{m \raro \infty} P^m_R(I_{\clh}) = 0,
\]
that is, $(R_1, \ldots, R_n)$ on $\clh$ is a pure row contraction.

At this point we pause to recall the dilation result due to Jewell
and Lubin \cite{JL} and Muller and Vasilescu \cite{MV} (also see
Arveson \cite{Ar}) which says that a pure row contraction is jointly
unitarily equivalent to the compression of the tuple of
multiplication operators by the coordinate functions $\{z_1, \ldots,
z_n\}$ on a vector-valued Drury-Arveson space to a joint
co-invariant subspace. In other words, the multiplication operator
tuple $(M_{z_1}, \ldots, M_{z_n})$  on the Drury-Arveson space plays
the role of the model pure row contraction. We state this more
formally as follows (see Theorem 3.1 \cite{JS2} for a proof):

\begin{Theorem}\label{DA-pure}
Let $\cll$ be a Hilbert space, and let $X = (X_1, \ldots, X_n)$ be a
pure row contraction on $\cll$. Then there exists a co-isometry
$\Pi_X : H^2_n(\cld_X) \raro \cll$ such that
\[
\Pi_X M_{z_j} = X_j \Pi_X,
\]
for all $j = 1, \ldots, n$.
\end{Theorem}

Therefore, by the above dilation theorem applied to the pure row
contraction $(R_1, \ldots, R_n)$, we get a co-isometry $\Pi_R :
H^2_n(\cld_R) \raro \clh$ such that
\[
\Pi_R M_{z_j} = R_j \Pi_R,
\]
for all $j = 1, \ldots, n$. Let
\[
\Pi = i_{\clh} \circ \Pi_{R}.
\]
It follows that $\Pi : H^2_n(\cld_R) \raro \clk$ is a contraction
and
\[
\mbox{ran~} \Pi = \clh.
\]
Moreover, since $i_{\clh} R_j = T_j i_{\clh}$, we have that
\[
\Pi M_{z_j} = T_j \Pi,
\]
for all $j = 1, \ldots, n$. We summarize these results as follows:

\begin{Theorem}\label{th:gen-inc}
Let $\clk$ be a Hilbert space, and let $(T_1, \ldots, T_n)$ be a
pure row contraction on $\clk$. Let $\clh$ be a Hilbert space that
is contractively embedded in $\clk$. Let $T_j \clh \subseteq \clh$
and
\[
R_j = T_j|_{\clh},
\]
for all $j = 1, \ldots, n$. Let $(R_1, \ldots, R_n)$ be a row
contraction on $\clh$. Then $(R_1, \ldots, R_n)$ is a pure row
contraction and there exist a Hilbert space $\cle_*$ and a
contraction $\Pi : H^2_n({\cle_*}) \raro \clk$ such that
\[
\Pi M_{z_j} = T_j \Pi,
\]
for all $j = 1, \ldots, n$, and
\[
\mbox{ran~} \Pi = \clh.
\]
\end{Theorem}

Of particular interest is the case where $(T_1, \ldots, T_n)$ is the
$n$-tuple of multiplication operators on a Hilbert space of analytic
functions in the unit ball. To this end, we first need to introduce
analytic Hilbert spaces over $\mathbb{B}^n$ (see \cite{JS2} and
\cite{JS1} for more details).

Let $K : \mathbb{B}^n \times \mathbb{B}^n \raro \mathbb{C}$ be a
positive definite kernel such that $K(\z, \w)$ is holomorphic in the
$\{z_1, \ldots, z_n\}$ variables and anti-holomorphic in $\{w_1,
\ldots, w_n\}$ variables. Then the corresponding reproducing kernel
Hilbert space $\clh_K$ is a Hilbert space of holomorphic functions
in $\mathbb{B}^n$. We say that $\clh_K$ is an \textit{analytic
Hilbert space} if $(M_{z_1}, \ldots, M_{z_n})$, the $n$-tuple of
multiplication operators by the coordinate functions $\{z_1, \ldots,
z_n\}$, defines a pure row contraction on $\clh_K$. In other words,
$M_{z_j}$ on $\clh_K$ defined by
\[
(M_{z_j} f)(\w) = w_j f(\w) \quad \quad (f \in \clh_K, \w \in
\mathbb{B}^n),
\]
is bounded for all $j = 1, \ldots, n$, the commuting tuple $M_z =
(M_{z_1}, \ldots, M_{z_n})$ on $\clh_K$ satisfies the positivity
condition
\[
\sum_{i=1}^n M_{z_i} M_{z_i}^* \leq I_{\clh_K},
\]
and
\[
P_{\infty}(M_z) = 0.
\]

Let $\cle$ be a Hilbert space. Consider the $\cle$-valued
reproducing kernel Hilbert space $\clh_{K_\lambda} \otimes \cle$,
$\lambda \geq 1$, where $K_\lambda$ is defined as in \eqref{eq-K}.
Then the reproducing kernel Hilbert space $\clh_{K_\lambda}  \otimes
\cle$ is analytic, as is well-known and also follows, for example,
from Proposition 4.1 in \cite{JS2}. In particular, the vector-valued
Drury-Arveson space $H^2_n  \otimes \cle$, the Hardy space
$H^2(\mathbb{B}^n) \otimes \cle$, the Bergman space
$L^2_a(\mathbb{B}^n)  \otimes \cle$, and the vector-valued weighted
Bergman spaces $L^2_{a, \alpha}(\mathbb{B}^n)  \otimes \cle$, with
$\alpha > 0$, are analytic Hilbert spaces.

We finally recall a characterization of intertwining maps between
vector-valued Drury-Arveson space and analytic Hilbert spaces (cf.
Proposition 4.2 in \cite{JS2}). Let $\cle_1$ and $\cle_2$ be Hilbert
spaces, $\clh_{K}$ be an analytic Hilbert space and let $X \in
\clb(H^2_n \otimes \cle_1, \clh_{K} \otimes \cle_2)$. Then
\[
X(M_{z_i} \otimes I_{\cle_1}) = (M_{z_i} \otimes I_{\cle_2}) X,
\]
for all $i = 1, \ldots, n$, if and only if there exists a multipler
$\Theta \in \clm(H^2_n \otimes \cle_1, \clh_{K} \otimes \cle_2)$
such that
\begin{equation}\label{eq-mtheta}
X = M_{\Theta}.
\end{equation}
Recall that the multiplier space $\clm(H^2_n \otimes \cle_1,
\clh_{K} \otimes \cle_2)$ is the Banach space of all operator-valued
analytic functions $\Theta : \mathbb{B}^n \raro \clb(\cle_1,
\cle_2)$ such that
\[
\Theta f \in \clh_{K} \otimes \cle_2,
\]
for all $f \in H^2_n \otimes \cle_1$. Note that if $\Theta \in
\clm(H^2_n \otimes \cle_1, \clh_{K} \otimes \cle_2)$, then the
multiplication operator $M_{\Theta}$ defined by
\[
(M_{\Theta} f)(\w) = \Theta(\w) f(\w),
\]
for all $f \in H^2_n \otimes \cle_1$ and $\w \in \mathbb{B}^n$, is a
bounded linear operator (by the closed graph theorem) from $H^2_n
\otimes \cle_1$ to $\clh_{K} \otimes \cle_2$ (cf. \cite{GRS},
\cite{MT} and \cite{JS2}). The next corollary now follows directly
from Theorem \ref{th:gen-inc}.

\begin{Theorem}\label{th-MT2}
Let $\cle_*$ be a Hilbert space, and let $\clh_K$ be an analytic
Hilbert space. Let $\cls$ be a Hilbert space that is contractively
embedded in $\clh_K \otimes \cle_*$. Let $M_{z_j} \cls \subseteq
\cls$ and
\[
R_j = M_{z_j}|_{\cls},
\]
for all $j = 1, \ldots, n$, and suppose that $(R_1, \ldots, R_n)$ is
a row contraction on $\cls$. Then $(R_1, \ldots, R_n)$ is a pure row
contraction and there exist a Hilbert space $\cle$ and a contractive
multiplier $\Theta \in \clm(H^2_n \otimes \cle, \clh_{K} \otimes
\cle_*)$ such that
\[
\cls = \Theta H^2_n(\cle).
\]
\end{Theorem}

In the case when $\clh_K$ is the Drury-Arveson space $H^2_n$, see
the early results in Benhida and Timotin (Theorem 4.2 \cite{BT}). In
this context we also refer to McCullough and Trent \cite{MT} and
Greene, Richter and Sundberg \cite{GRS}.

We would like to point out that the theory of contractively embedded
backward shift invariant subspaces in reproducing kernel Hilbert
spaces and the de Branges-Rovnyak models, in the setting of row
contractions, are closely related to the Gleason's problem
\cite{AD2}. In this context, the reader should consult the papers by
Alpay and Dubi \cite{AD1}, Ball and Bolotnikov \cite{BB10}, Ball,
Bolotnikov and Fang \cite{BBF2, BBF}, Ball, Bolotnikov and ter Horst
\cite{BBS1, BBS}, Benhida and Timotin \cite{BT} and Martin and
Ramanantoanina \cite{Ma}.

\newsection{Hardy space over the polydisc}

Let $n$ be a natural number. Given a Hilbert space $\cle$, we denote
by $H^2_{\cle}(\D^{n+1})$ the $\cle$-valued Hardy space over the
polydisc $\D^{n+1}$. In this section we aim to analyze the structure
of contractively embedded invariant subspaces for the multiplication
tuple on $H^2_{\cle}(\D^{n+1})$. The principle of our method is
based on the idea \cite{MMSS} that one can represent the tuple of
shifts on the Hardy space over $\D^{n+1}$ by a natural $(n+1)$-tuple
of multiplication operators on a vector-valued Hardy space over the
unit disc. This is the main content of the following theorem (see
Theorem 3.1 in \cite{MMSS}).

\begin{Theorem}\label{thm-old1}
Let $n$ be a natural number, and let $\cle$ be a Hilbert space. Let
\[
\cle_n = H^2_{\cle}(\D^n).
\]
For each $i = 1, \ldots, n$, let $\kappa_i \in
H^\infty_{\clb(\cle_n)}(\D)$ denote the $\clb(\cle_n)$-valued
constant function on $\D$ defined by
\[
\kappa_i(w) = M_{z_i} \in \clb(\cle_n),
\]
for all $w \in \D$, and let $M_{\kappa_i}$ denote the multiplication
operator on $H^2_{\cle_n}(\D)$ defined by
\[
M_{\kappa_i} f = \kappa_i f,
\]
for all $f \in H^2_{H_n}(\D)$. Then the $(n+1)$-tuples $(M_{z_1},
M_{z_2} \ldots, M_{z_{n+1}})$ and $(M_z, M_{\kappa_1}, \ldots,
M_{\kappa_n})$ are unitarily equivalent.
\end{Theorem}

\begin{proof}
We briefly sketch only the main ideas behind the proof and refer the
reader to the proof of Theorem 3.1 in \cite{MMSS} for details. Since
the linear spans of
\[
\{z_1^{k_1} z_2^{k_2} \cdots z_{n+1}^{k_{n+1}} \eta: k_1, \ldots,
k_{n+1} \geq 0, \eta \in \cle\} \subseteq H^2_{\cle}(\D^{n+1}),
\]
and
\[
\{z^{k} (z_1^{k_1} \cdots z_{n}^{k_{n}} \eta): k, k_1, \ldots, k_{n}
\geq 0, \eta \in \cle\} \subseteq H^2_{\cle_n}(\D),
\]
are dense in $H^2_{\cle}(\D^{n+1})$ and $H^2_{\cle_n}(\D)$,
respectively, it follows that the map  ${U} : H^2_{\cle}(\D^{n+1})
\raro H^2_{\cle_n}(\D)$ defined by
\[
{U}(z_1^{k_1} z_2^{k_2} \cdots z_{n+1}^{k_{n+1}} \eta) = z^{k_1}
(z_1^{k_2} \cdots z_{n}^{k_{n+1}} \eta),
\]
for all $k_1, \ldots, k_{n+1} \geq 0$ and $\eta \in \cle$, is a
unitary operator. Clearly
\[
U M_{z_1} = M_z U,
\]
and an easy computation yields
\[
U M_{z_{i}} = M_{\kappa_{i-1}} U,
\]
for all $i = 2, \ldots, n$. This completes the proof.
\end{proof}

In view of the above theorem, we can now consider the problem of
contractively embedded invariant subspaces for the tuple $(M_z,
M_{\kappa_1}, \ldots, M_{\kappa_n})$ on $H^2_{\cle_n}(\D)$ instead
of the tuple of multiplication operators $(M_{z_1}, M_{z_2} \ldots,
M_{z_{n+1}})$ on the vector-valued Hardy space
$H^2_{\cle}(\D^{n+1})$.

Before we proceed to the main result of this section, we need one
more result concerning representations of commutators of shift
\cite{H} operators. Here our approach follows that of \cite{MMSS}
and \cite{MSS}. Recall that an isometry $V$ on a Hilbert space
$\clh$ is said to be a \textit{shift} if
\[
SOT - \lim_{m \raro \infty} V^{*m} = 0,
\]
that is, $\|V^{*m} f \| \raro 0$ as $m \raro \infty$ for all $f \in
\clh$, or equivalently, if there is no non trivial reducing subspace
of $\clh$ on which $V$ is unitary. Now, if $V$ is a shift on $\clh$,
then
\[
\clh = \mathop{\oplus}_{m=0}^\infty V^m \clw,
\]
where $\clw = \ker V^* = \clh \ominus V \clh$ is the wandering
subspace for $V$. By the above decomposition of $\clh$, we see that
the map $\Pi_V : \clh \raro H^2_{\clw}(\D)$ defined by
\[
\Pi_V (V^m \eta) = z^m \eta,
\]
for all $m \geq 0$ and $\eta\in \clw$, is a unitary operator and
\[
\Pi_V V = M_z \Pi_V.
\]
Following Wold and von Neumann, we call $\Pi_V$ the \textit{Wold-von
Neumann decomposition} of the shift $V$ (see \cite{MMSS} and
\cite{MSS}).

This point of view is very useful in representing the commutators of
shifts (see Theorem 2.1 in \cite{MSS} and Theorem 2.1 in
\cite{MMSS}):

\begin{Theorem}\label{thm-old2}(Theorem 2.1 in \cite{MSS})
Let $\clh$ be a Hilbert space. Let $V$ be a shift on $\clh$, and let
$C$ be a bounded operator on $\clh$. Let $\Pi_V$ be the Wold-von
Neumann decomposition of $V$, $M = \Pi_V C \Pi^{*}_V$, and let
\[
\Theta(w)= P_{\clw}(I_{\clh} - w V^{*})^{-1}C\mid_{\clw},
\]
for all $w \in \D$. Then $C V = VC$ if and only if $\Theta \in
H^\infty_{\clb(\clw)}(\D)$ and
\[
M = M_{\Theta}.
\]
\end{Theorem}

Since $\|w V^*\| = |w| \|V\| < 1$ for all $w \in \D$, it follows
that, given a bounded operator $C$ on $\clw$, the function $\Theta$
as defined above is a $\clb(\clw)$-valued analytic function on $\D$.
It is however not clear that $\Theta$ is a bounded function on $\D$,
that is, $\Theta \in H^\infty_{\clb(\clw)}(\D)$. The above theorem
says that this is so if and only if $C$ is in the commutator of $V$.

\NI\textit{Proof of Theorem \ref{thm-old2}:} Again we will only
sketch the proof and refer the reader to Theorem 2.1 in \cite{MSS}
for a more rigorous proof. Certainly, the sufficient part follows
from the representation of $C$ (as $\Pi_V^* M_{\Theta} \Pi_V = C$)
and the fact that $M_z M_{\Theta} = M_{\Theta} M_z$. The proof for
the necessary part relies on the fact that (cf. \cite{MSS})
\[
I_{\clh} = \sum_{m=0}^\infty V^m P_{\clw} V^{*m},
\]
in the strong operator topology. Indeed, if $CV = VC$, then $M M_z =
M_z M$, and so
\[
M = M_{\Theta},
\]
for some bounded analytic function $\Theta \in
H^{\infty}_{\clb(\clw)}(\D)$ (see, for instance, the equality in
\eqref{eq-mtheta}). Let $w \in \D$ and $\eta \in \clw$. Then
\[
\Theta(w) \eta  = (M_{\Theta} \eta)(w) = (\Pi_V C \Pi_V^* \eta)(w).
\]
Since $\Pi_V^* \eta = \eta$ and
\[
C \eta = \sum_{m=0}^{\infty} V^m P_{\clw} V^{*m} C \eta,
\]
it follows that
\[
\begin{split}
\Theta(w) \eta & = (\Pi_V C \eta)(w)
\\
& = (\Pi_V (\sum_{m=0}^{\infty} V^m P_{\clw} V^{*m} C \eta))(w)
\\
& = (\sum_{m=0}^{\infty} M_z^m (P_{\clw} V^{*m} C \eta))(w).
\end{split}
\]
Finally, note that $P_{\clw} V^{*m} C \eta \in \clw$ for all $m \geq
0$, and hence
\[
\Theta(w) \eta = \sum_{m=0}^{\infty} w^m (P_{\clw} V^{*m} C \eta),
\]
from which the result follows.
\qed

We are now ready for the main result concerning contractively
embedded invariant subspaces of vector-valued Hardy spaces.

Let $n$ be a natural number, and let $\cle$ be a Hilbert space. Let
$\cls$ be a Hilbert space that is contractively embedded in
$H^2_{\cle_n}(\D)$. Let
\[
z \cls \subseteq \cls,
\]
and
\[
\kappa_i \cls \subseteq \cls,
\]
for all $i = 1, \ldots, n$. Assume that $(R, R_1, \ldots, R_n)$ is
an $(n+1)$-tuple of isometries on $\cls$, where
\[
R = M_z|_{\cls},
\]
and
\[
R_i = M_{\kappa_i}|_{\cls},
\]
for all $i = 1, \ldots, n$. We have
\[
\begin{split}
\mathop{\cap}_{m  =0}^\infty R^m \cls & = \mathop{\cap}_{m=0}^\infty
z^m \cls \subseteq \mathop{\cap}_{m=0}^\infty z^m H^2_{\cle_n}(\D).
\end{split}
\]
But $M_z$ on $H^2_{\cle_n}(\D)$ is a pure isometry (shift), that is,
\[
\mathop{\cap}_{m=0}^\infty z^m H^2_{\cle_n}(\D) = \{0\},
\]
and so it follows that
\[
\mathop{\cap}_{m=0}^\infty R^m \cls = \{0\}.
\]
Further, since $M_{\kappa_j}$ is a shift, it follows that
\[
\mathop{\cap}_{m=0}^\infty \kappa_j^m H^2_{\cle_n}(\D) = \{0\},
\]
and so
\[
\mathop{\cap}_{m=0}^\infty R_j^m \cls = \{0\},
\]
for all $j = 1, \ldots, n$, follows in a similar way. In other
words, $(R, R_1, \ldots, R_n)$ is an $(n+1)$-tuple of commuting
shifts on $\cls$. Now we argue essentially as in the proof of
Theorem 3.2 in \cite{MMSS}. Let
\[
\clw = \cls \ominus z \cls,
\]
and let $\Pi_R : \cls \raro H^2_{\clw}(\D)$ be the Wold-von Neumann
decomposition of $R$ on $\cls$. In particular, we have
\begin{equation}\label{eq=R}
R \Pi_R^* = \Pi_R^* M_z.
\end{equation}
Moreover, since $R R_j = R_j R$, applying Theorem \ref{thm-old2}, we
have
\begin{equation}\label{eq=Rj}
\Pi_R R_j = M_{\Phi_j} \Pi_R,
\end{equation}
where the $\clb(\clw)$-valued analytic function defined by
\[
\Phi_j(w) = P_{\clw} (I_{\cls} - P_{\cls} M_z^*)^{-1}
M_{\kappa_j}|_{\clw},
\]
for all $w \in \D$, is in $H^\infty_{\clb(\clw)}(\D)$ and $j = 1,
\ldots, n$. Now consider the (contractive) inclusion map $i_{\cls} :
\cls \hookrightarrow H^2_{\cle_n}(\D)$. Set
\[
\Pi = i_{\cls} \circ \Pi_R^*.
\]
Then $\Pi : H^2_{\clw}(\D) \raro H^2_{\cle_n}(\D)$ is a contraction.
Moreover, since
\[
i_{\cls} R = M_z i_{\cls},
\]
and
\[
i_{\cls} R_j = M_{\kappa_j} i_{\cls},
\]
it follows from \eqref{eq=R} and \eqref{eq=Rj} that
\begin{equation}\label{eq=Pi}
\Pi M_z = M_z \Pi,
\end{equation}
and
\begin{equation}\label{eq=Pij}
\Pi M_{\Phi_j} = M_{\kappa_j} \Pi,
\end{equation}
for all $j = 1, \ldots, n$. Then using \eqref{eq-mtheta}, one sees
that
\[
\Pi = M_{\Theta},
\]
for some contractive multiplier $\Theta \in H^\infty_{\clb(\clw,
\cle_n)}(\D)$, from \eqref{eq=Pi}, and hence
\[
\Theta \Phi_j = \kappa_j \Theta,
\]
from \eqref{eq=Pij}, for all $j = 1, \ldots, n$. Since $\text{ran~}
i_{\cls} = \cls$, it follows from the definition of $\Pi$ that
\[
\cls = \Theta H^2_{\clw}(\D).
\]
We can therefore state the following analogue of the de Branges
theorem in the setting of Hardy space over the unit polydisc:

\begin{Theorem}\label{th-MT3}
Let $n$ be a natural number, and let $\cle$ be a Hilbert space. Let
$\cls$ be a Hilbert space that is contractively embedded in
$H^2_{\cle_n}(\D)$. Let $z \cls \subseteq \cls$ and
\[
R = M_z|_{\cls}.
\]
For each $j = 1, \ldots, n$, let $\kappa_{j} \cls \subseteq \cls$
and
\[
R_j = M_{\kappa_j}|_{\cls}.
\]
Set $\clw = \cls \ominus z \cls$ and
\[
\Phi_j(w) = P_{\clw}(I_{\cls} - w P_{\cls} M_z^*)^{-1}
M_{\kappa_j}|_{\clw},
\]
for all $w \in \D$ and $j = 1, \ldots, n$. If $(R, R_1, \ldots,
R_n)$ is an $(n+1)$-tuple of commuting isometries on $\cls$, then
$(M_{\Phi_1}, \ldots, M_{\Phi_n})$ is an $n$-tuple of commuting
shifts on $H^2_{\clw}(\D)$ and there exists a contractive multiplier
$\Theta \in H^\infty_{\clb(\clw, \cle_n)}(\D)$ such that
\[
\cls = \Theta H^2_{\clw}(\D),
\]
and
\[
\kappa_j \Theta = \Theta \Phi_j,
\]
for all $j = 1, \ldots, n$.
\end{Theorem}

The preceding result, in view of the de Branges and Rovnyak theory,
suggests a very interesting question: How can one characterize those
contractively embedded invariant subspace for $(M_z, M_{\kappa_1},
\ldots, M_{\kappa_n})$ on $H^2_{\cle_n}(\D)$ which are complemented
by invariant subspace for $(M_z^*, M_{\kappa_1}^*, \ldots,
M_{\kappa_n}^*)$ on $H^2_{\cle_n}(\D)$? The answer to this question
is not known.

\bigskip

\noindent\textbf{Acknowledgment:} The first named author's research
work is supported by an INSPIRE faculty fellowship (IFA-MA-02)
funded by DST. The second author is supported in part by NBHM
(National Board of Higher Mathematics, India) grant
NBHM/R.P.64/2014.

\end{document}